\documentclass{amsart}
\usepackage{amssymb}
\usepackage{amscd}
\usepackage{verbatim}
\usepackage{epsfig}

\begin{document}
\newcommand\Mand{\ \text{and}\ }
\newcommand\Mor{\ \text{or}\ }
\newcommand\Mfor{\ \text{for}\ }
\newcommand\Real{\mathbb{R}}
\newcommand\RR{\mathbb{R}}
\newcommand\im{\operatorname{Im}}
\newcommand\re{\operatorname{Re}}
\newcommand\sign{\operatorname{sign}}
\newcommand\sphere{\mathbb{S}}
\newcommand\BB{\mathbb{B}}
\newcommand\HH{\mathbb{H}}
\newcommand\dS{\mathrm{dS}}
\newcommand\ZZ{\mathbb{Z}}
\newcommand\codim{\operatorname{codim}}
\newcommand\Sym{\operatorname{Sym}}
\newcommand\End{\operatorname{End}}
\newcommand\Span{\operatorname{span}}
\newcommand\Ran{\operatorname{Ran}}
\newcommand\ep{\epsilon}
\newcommand\Cinf{\cC^\infty}
\newcommand\dCinf{\dot \cC^\infty}
\newcommand\CI{\cC^\infty}
\newcommand\dCI{\dot \cC^\infty}
\newcommand\Cx{\mathbb{C}}
\newcommand\Nat{\mathbb{N}}
\newcommand\dist{\cC^{-\infty}}
\newcommand\ddist{\dot \cC^{-\infty}}
\newcommand\pa{\partial}
\newcommand\Card{\mathrm{Card}}
\renewcommand\Box{{\square}}
\newcommand\Ell{\mathrm{Ell}}
\newcommand\WF{\mathrm{WF}}
\newcommand\WFh{\mathrm{WF}_\semi}
\newcommand\WFb{\mathrm{WF}_\bl}
\newcommand\WFsc{\mathrm{WF}_\scl}
\newcommand\Vf{\mathcal{V}}
\newcommand\Vb{\mathcal{V}_\bl}
\newcommand\Vsc{\mathcal{V}_\scl}
\newcommand\Vz{\mathcal{V}_0}
\newcommand\Hb{H_{\bl}}
\newcommand\Hsc{H_{\scl}}
\newcommand\Ker{\mathrm{Ker}}
\newcommand\Range{\mathrm{Ran}}
\newcommand\Hom{\mathrm{Hom}}
\newcommand\Id{\mathrm{Id}}
\newcommand\sgn{\operatorname{sgn}}
\newcommand\ff{\mathrm{ff}}
\newcommand\tf{\mathrm{tf}}
\newcommand\esssupp{\operatorname{esssupp}}
\newcommand\supp{\operatorname{supp}}
\newcommand\vol{\mathrm{vol}}
\newcommand\Diff{\mathrm{Diff}}
\newcommand\Diffd{\mathrm{Diff}_{\dagger}}
\newcommand\Diffs{\mathrm{Diff}_{\sharp}}
\newcommand\Diffb{\mathrm{Diff}_\bl}
\newcommand\Diffsc{\mathrm{Diff}_\scl}
\newcommand\DiffbI{\mathrm{Diff}_{\bl,I}}
\newcommand\Diffbeven{\mathrm{Diff}_{\bl,\even}}
\newcommand\Diffz{\mathrm{Diff}_0}
\newcommand\Psih{\Psi_{\semi}}
\newcommand\Psihcl{\Psi_{\semi,\cl}}
\newcommand\Psisc{\Psi_\scl}
\newcommand\Psiscc{\Psi_\sccl}
\newcommand\Psib{\Psi_\bl}
\newcommand\Psibc{\Psi_{\mathrm{bc}}}
\newcommand\TbC{{}^{\bl,\Cx} T}
\newcommand\Tb{{}^{\bl} T}
\newcommand\Sb{{}^{\bl} S}
\newcommand\Tsc{{}^{\scl} T}
\newcommand\Ssc{{}^{\scl} S}
\newcommand\Lambdab{{}^{\bl} \Lambda}
\newcommand\zT{{}^{0} T}
\newcommand\Tz{{}^{0} T}
\newcommand\zS{{}^{0} S}
\newcommand\dom{\mathcal{D}}
\newcommand\cA{\mathcal{A}}
\newcommand\cB{\mathcal{B}}
\newcommand\cE{\mathcal{E}}
\newcommand\cG{\mathcal{G}}
\newcommand\cH{\mathcal{H}}
\newcommand\cU{\mathcal{U}}
\newcommand\cO{\mathcal{O}}
\newcommand\cF{\mathcal{F}}
\newcommand\cM{\mathcal{M}}
\newcommand\cQ{\mathcal{Q}}
\newcommand\cR{\mathcal{R}}
\newcommand\cI{\mathcal{I}}
\newcommand\cL{\mathcal{L}}
\newcommand\cK{\mathcal{K}}
\newcommand\cC{\mathcal{C}}
\newcommand\cX{\mathcal{X}}
\newcommand\cY{\mathcal{Y}}
\newcommand\cP{\mathcal{P}}
\newcommand\cS{\mathcal{S}}
\newcommand\cZ{\mathcal{Z}}
\newcommand\cW{\mathcal{W}}
\newcommand\Ptil{\tilde P}
\newcommand\ptil{\tilde p}
\newcommand\chit{\tilde \chi}
\newcommand\yt{\tilde y}
\newcommand\zetat{\tilde \zeta}
\newcommand\xit{\tilde \xi}
\newcommand\taut{{\tilde \tau}}
\newcommand\phit{{\tilde \phi}}
\newcommand\mut{{\tilde \mu}}
\newcommand\sigmat{{\tilde \sigma}}
\newcommand\sigmah{\hat\sigma}
\newcommand\zetah{\hat\zeta}
\newcommand\etah{\hat\eta}
\newcommand\loc{\mathrm{loc}}
\newcommand\compl{\mathrm{comp}}
\newcommand\reg{\mathrm{reg}}
\newcommand\GBB{\textsf{GBB}}
\newcommand\GBBsp{\textsf{GBB}\ }
\newcommand\bl{{\mathrm b}}
\newcommand\scl{{\mathrm{sc}}}
\newcommand\sccl{{\mathrm{scc}}}
\newcommand{\sH}{\mathsf{H}}
\newcommand{\cte}{\digamma}
\newcommand\cl{\operatorname{cl}}
\newcommand\hsf{\mathcal{S}}
\newcommand\Div{\operatorname{div}}
\newcommand\hilbert{\mathfrak{X}}
\newcommand\smooth{\mathcal{J}}
\newcommand\decay{\ell}
\newcommand\symb{j}

\newcommand\Hh{H_{\semi}}

\newcommand\bM{\bar M}
\newcommand\Xext{X_{-\delta_0}}

\newcommand\xib{{\underline{\xi}}}
\newcommand\etab{{\underline{\eta}}}
\newcommand\zetab{{\underline{\zeta}}}

\newcommand\xibh{{\underline{\hat \xi}}}
\newcommand\etabh{{\underline{\hat \eta}}}
\newcommand\zetabh{{\underline{\hat \zeta}}}

\newcommand\zn{z}
\newcommand\sigman{\sigma}
\newcommand\psit{\tilde\psi}
\newcommand\rhot{{\tilde\rho}}

\newcommand\hM{\hat M}

\newcommand\Op{\operatorname{Op}}
\newcommand\Oph{\operatorname{Op_{\semi}}}

\newcommand\innr{{\mathrm{inner}}}
\newcommand\outr{{\mathrm{outer}}}
\newcommand\full{{\mathrm{full}}}
\newcommand\semi{\hbar}

\newcommand\Feynman{\mathrm{Feynman}}
\newcommand\future{\mathrm{future}}
\newcommand\past{\mathrm{past}}

\newcommand\elliptic{\mathrm{ell}}
\newcommand\diffordgen{k}
\newcommand\difford{2}
\newcommand\diffordm{1}
\newcommand\diffordmpar{1}
\newcommand\even{\mathrm{even}}
\newcommand\dimn{n}
\newcommand\dimnpar{n}
\newcommand\dimnm{n-1}
\newcommand\dimnp{n+1}
\newcommand\dimnppar{(n+1)}
\newcommand\dimnppp{n+3}
\newcommand\dimnppppar{n+3}

\newcommand\sob{s}

\newtheorem{lemma}{Lemma}
\newtheorem{prop}[lemma]{Proposition}
\newtheorem{thm}[lemma]{Theorem}
\newtheorem{cor}[lemma]{Corollary}
\newtheorem{result}[lemma]{Result}
\newtheorem*{thm*}{Theorem}
\newtheorem*{prop*}{Proposition}
\newtheorem*{cor*}{Corollary}
\newtheorem*{conj*}{Conjecture}
\theoremstyle{remark}
\newtheorem{rem}[lemma]{Remark}
\newtheorem*{rem*}{Remark}
\theoremstyle{definition}
\newtheorem{Def}[lemma]{Definition}
\newtheorem*{Def*}{Definition}

\newcommand{\mar}[1]{{\marginpar{\sffamily{\scriptsize #1}}}}
\newcommand\av[1]{\mar{AV:#1}}

\renewcommand{\theenumi}{\roman{enumi}}
\renewcommand{\labelenumi}{(\theenumi)}

\title[Essential self-adjointness of the wave operator]{Essential self-adjointness of the wave operator and the
  limiting absorption principle on Lorentzian scattering spaces}
\author[Andras Vasy]{Andr\'as Vasy}
\address{Department of Mathematics, Stanford University, CA 94305-2125, USA}

\email{andras@math.stanford.edu}

\date{December 27, 2017}


\thanks{The author is very grateful to Jan Derezi{\'n}ski for
  suggesting this line of investigation and for very interesting
  discussions and comments on an earlier version of the manuscript. He is also grateful to Shu Nakamura, Kouichi Taira and
  Micha{\l} Wrochna for stimulating discussions on this matter, and
  also to Micha{\l} Wrochna for a careful reading of an earlier
  version of the manuscript. The
  author thanks the NSF for partial support under grant number
  DMS-1361432, the Simons Foundation for partial support via a Simons
  fellowship grant, and the Kyoto University RIMS for support during
  the conference at which the final conceptual ingredients of this project were completed.}

\begin{abstract}
We discuss the essential self-adjointness of wave operators, as well as the limiting absorption principle, in
generalizations of asymptotically
Minkowski settings. This
is obtained via using a Fredholm framework for inverting the spectral
family first, and then refining its conclusions to show dense range of
$\Box-\lambda$, $\lambda\notin\RR$, in $L^2_\scl$ when acting on an appropriate subdomain.
\end{abstract}

\maketitle

\section{Introduction}
In this short note we discuss the self-adjointness of the wave
operator on generalizations of Minkowski space, answering a question
of Jan Derezi\'nski. More precisely, the setting is that of
non-trapping sc-metrics, an extension of Lorentzian scattering metrics introduced in
\cite{Baskin-Vasy-Wunsch:Radiation} and studied in more detail in
\cite{Hintz-Vasy:Semilinear},
\cite{Gell-Redman-Haber-Vasy:Feynman} and \cite{Vasy-Wrochna:QFT}, with the Feynman propagator,
whose role is discussed below,
being particularly closely examined in the latter two papers. These are
Lorentzian analogues of the Riemannian scattering metrics introduced
by Melrose in \cite{RBMSpec}. The non-trapping condition is a
condition on null-geodesics on $M$, namely they should converge to
a replacement of the light
cone at infinity in both the forward and backward directions. In fact, the signature of the metric makes no difference; the same  
conclusion is true for non-trapping semi-Riemannian metrics of any  
signature, as the proof goes through without any changes.
Later on, we also discuss the limiting absorption principle, for which
there is also a `non-trapping at
energy $\lambda$' condition which is a condition on the
limiting geodesics at infinity corresponding to the spectral parameter
being considered, see Section~\ref{sec:background} for detail.

For the statement of the first result recall
that for a (densely defined) unbounded operator $L$ self-adjointness is a symmetry {\em
  plus} an invertibility (for the operator $L\pm\iota$), or
equivalently surjectivity statement; essential self-adjointness
amounts to a symmetry plus a dense range statement for the operator $L\pm\iota$.

\begin{thm}\label{thm:ess-sa}
Suppose $(M,g)$ is a non-trapping Lorentzian sc-metric. Then $\Box_g$ is
essentially self-adjoint on $\CI_c(M^\circ)$.
\end{thm}

As far as the author is aware, the first general mathematically
precise result in this
direction is that of Derezi{\'n}ski and Siemssen
\cite{Derezinski-Siemssen:Feynman}, see also
\cite{Derezinski-Siemssen:Evolution}, who assumed time-translation
invariance, though there is a long history in the physics literature
of treating the wave operator as at least a potentially self-adjoint
operator. Note that this time translation invariance means that the
overlap of the present paper with \cite{Derezinski-Siemssen:Feynman}
is minimal.
Our result also relates to recent/ongoing work of Nakamura and Taira,
see \cite{Taira:Strichartz}, with details of the directly relevant
aspects being parts of works in progress.
In the work of
Nakamura and Taira \cite{Taira:Strichartz}, the non-trapping condition (for the
purpose of essential self-adjointness) is replaced by a positive
energy condition.

The key part of proving the theorem is to show that $(\Box_g\pm
\imath) u=f$ is solvable when $f\in\CI_c(M^\circ)$. Concretely, we take
$$
D=\{u\in\Hsc^{1,-1/2}\cap L^2_\scl:\Box_g u\in
L^2_\scl\},
$$
the weighted scattering Sobolev spaces
being recalled below, and then
a straightforward regularization argument shows that $\CI_c(M^\circ)$ is dense in it and $\Box_g$ is
symmetric on it. The main point is thus to show that $\CI_c(M^\circ)\subset (\Box_g\pm \iota)
D\subset L^2_{\scl}$ and hence $(\Box_g\pm \iota)D$ is dense in $L^2_\scl$.
We do so by using a Fredholm framework for
inverting $\Box_g-\lambda$ on appropriate variable order Sobolev
spaces discussed below; this in fact works uniformly to the real
axis in $\lambda$, thus giving the limiting absorption
principle. We then show additional regularity of the solution, proving the Theorem~\ref{thm:ess-sa}.

The aforementioned Fredholm framework gives rise to the massive
Feynman propagators via the limiting absorption principle.
G\'erard 
and Wrochna studied these in a different Fredholm setting in 
\cite{Gerard-Wrochna:Hadamard,Gerard-Wrochna:Massive}, based in part 
on earlier work of B\"ar and Strohmaier \cite{Baer-Strohmaier:Index}.

We finally comment on some generalizations. Considering
electromagnetic potentials $A$ amounts to working with
$(-\imath d-A)_g^*(-\imath d- A)+V$. If $A,V$ are real and symbols of
negative order (thus decaying) with
values in one-forms, resp.\ scalars, all our results and arguments are unaffected. If
$A,V$ are real and symbols of order $0$ then essential
self-adjointness (including its proof) is unaffected.

Working with the
wave operator on differential forms of any (or all) form degree or on
tensors again does not affect the Fredholm theory; again adding to
these symmetric (on $\CI_c(M^\circ)$ with values in forms or tensors,
with the $L^2$-inner product) first order operators with symbolic of negative order
coefficients (such as the electromagnetic terms above) does not affect
the Fredholm theory either. In particular, the limiting absorption
principle holds in the sense of Theorem~\ref{thm:LAP}, i.e.\ a
limiting resolvent exists, assuming that at $\lambda\in\RR$ the a
priori finite dimensional
nullspace of $P=\Box_g-\lambda$ and $P^*$ is trivial. ({\em Both} of
these have to be assumed to be trivial: in Theorem~\ref{thm:LAP},
index $0$ considerations mean that only one of these need to be assumed.) However, the argument for showing the
triviality of the a priori finite dimensional nullspace of
$P=\Box_g-\lambda$, $\lambda\notin\RR$, and its adjoint in
Lemma~\ref{lemma:trivial-nullspace} would be affected since the inner
product with respect to which the operator is symmetric is no longer
positive definite. However, if the operator has additional structure,
the nullspace may be shown to be trivial by other arguments, e.g.\
Wick rotations work in the case of translation invariant metrics on
$\RR^n$; the perturbation stability of the Fredholm framework implies
the same conclusion on perturbations of these in the scattering
category. Hence, in these cases (when the conclusion of
Lemma~\ref{lemma:trivial-nullspace} holds), the essential
self-adjointness also holds.

\section{Background}\label{sec:background}
We now recall some background. We refer to \cite{RBMSpec} for the
introduction of scattering, or sc-, structures, and to
\cite{Vasy:Propagation-Notes,Vasy:Minicourse} for another discussion
which emphasizes an $\RR^n$-based perspective localizing to asymptotic
cones.
Recall that on a manifold with boundary $M$, the space of {\em b-vector fields} $\Vb(M)$ is the Lie
algebra of smooth vector fields tangent to $\pa M$ (indeed, this is
the definition even for manifolds with corners, which will be used
below for the compactified cotangent bundle), while the space of
{\em scattering vector fields or sc-vector fields} is
$\Vsc(M)=x\Vb(M)$, where $x$ is any {\em boundary defining function}, i.e.\
a non-negative $\CI$ function on $M$, with zero set exactly $\pa M$
such that $dx$ is non-zero at $\pa M$. Such vector fields are exactly
all smooth sections of a vector bundle, $\Tsc M$, over $M$, called the
{\em scattering tangent bundle}, which over
the interior $M^\circ$ is naturally identified with $T
M^\circ$. Indeed, notice that in a local coordinate chart, in which
$x$ is one of the coordinates, and the other coordinates (coordinates
on $\pa M$) are $y_1,\ldots,y_{n-1}$, $V\in\Vsc(M)$ means exactly that
$V=a_0(x^2\pa_x)+\sum_{j=1}^{n-1}a_j (x\pa_{y_j})$ with $a_j$ smooth
in the chart, so $x^2\pa_x$, $x\pa_{y_1},\ldots,x\pa_{y_{n-1}}$ give a
local basis of smooth sections, and thus a local basis for the fibers
of the vector bundle $\Tsc M$. Hence, the $a_j$ are coordinates on the
fibers of $\Tsc M$ (locally), and thus
$x,y_1,\ldots,y_{n-1},a_0,a_1,\ldots,a_{n-1}$ are local coordinates on the
bundle $\Tsc M$. There is a dual vector bundle,
$\Tsc^*M$, called the {\em scattering cotangent bundle}, with local basis
$\frac{dx}{x^2},\frac{dy_1}{x},\ldots,\frac{dy_{n-1}}{x}$. A
{\em sc-metric of signature $(k,n-k)$} is then a smooth (in the base
point $p$) non-degenerate symmetric bilinear map $\Tsc_p M\times\Tsc_p
M\to \RR$ of signature $(k,n-k)$. Equivalently, it is a smooth section
of $\Tsc^*M\otimes_s\Tsc^*M$ (symmetric tensor product) of the
appropriate signature.
Then $L^2_\scl$ is the $L^2$-space of
the metric density of any sc-metric (either Lorentzian or
Riemannian, or of another definite signature), with all choices being
equivalent in that they define the same space and equivalent norms, and $\Hsc^{s,r}$ is the
corresponding weighted Sobolev space, so if $s\geq 0$ integer then
$$
\Hsc^{s,0}=\{u\in L^2_\scl(M):\ \forall k\leq s\  \forall V_1,\ldots
V_k\in\Vsc(M),\ V_1\ldots V_k u\in L^2_{\scl}\},
$$
and $\Hsc^{s,r}(M)=x^r\Hsc^{s,0}(M)$.

As an example, $M$ could be the {\em radial compactification}
$\overline{\RR^n}$ of $\RR^n$,
in which a sphere $\sphere^{n-1}$ is attached as the ideal boundary of
$\RR^n$, so the compactification is diffeomorphic to a closed
ball. Concretely, a neighborhood of the boundary is diffeomorphic to
$[0,\ep)_x\times\sphere^{n-1}$, $\ep>0$, whose interior,
$(0,\ep)_x\times\sphere^{n-1}$ is identified with the subset
$\{z\in\RR^n:\ |z|>\ep^{-1}\}$ via the reciprocal spherical coordinate
map, $(0,\ep)\times\sphere^{n-1}\ni (x,\omega)\mapsto
x^{-1}\omega\in\RR^n$, where the sphere is regarded as a submanifold
of $\RR^n$ to make sense of the map.
Then any translation invariant metric of any signature is (i.e.\
can be naturally identified with) a
sc-metric of the same signature. In fact, $\CI(M)$ is then the space
of {\em classical (one-step polyhomogeneous) symbols of order $0$} on $\RR^n$,
$\dCI(M)$ (the space of $\CI$ functions on $M$ vanishing to infinite
order at $\pa M$) is the space of Schwartz functions $\cS(\RR^n)$,
$\Vsc(M)$ is spanned by the lift of translation invariant vector
fields $\pa_{z_j}$, $j=1,\ldots,n$, over $\CI(M)$, i.e.\ an element of
$\Vsc(M)$ is of the form $\sum a_j\pa_{z_j}$ with $a_j\in\CI(M)$,
i.e.\ a classical symbol of order $0$, and similarly
$\Tsc^*M\otimes_s\Tsc^*M$ is spanned by the lifts of $dz_i\otimes_s
dz_j$ with $\CI(M)$ coefficients. Correspondingly, $L^2_\scl$ is the
standard Lebesgue space, $\Hsc^{s,r}$ the standard weighted Sobolev
space.

More generally, a coordinate neighborhood of a point on the
boundary of a manifold with boundary $M$ can be identified with a
similar coordinate neighborhood of a point on the boundary of
$\overline{\RR^n}$. Note that from the perspective of $\RR^n$, such a
neighborhood is asymptotically conic. Rather than following the above
intrinsic definitions, one could transplant the notions discussed
above directly from $\RR^n$ via such an identification, {\em exactly}
how standard notions on manifolds without boundary are defined by
identifying coordinate charts with open subsets of $\RR^n$; this is
the approach taken by \cite{Vasy:Minicourse, Vasy:Propagation-Notes}.

In particular, this gives a convenient way of introducing scattering
pseudodifferential operators $\Psiscc^{m,l}(M)$ by reducing to the case
of $\overline{\RR^n}$, or equivalently to an appropriate uniform
structure in the interior, $\RR^n$. (The notation $\Psiscc$ stands
for `scattering conormal'; \cite{RBMSpec} uses $\Psisc$ for
`classical' scattering; classical
symbols are the one-step polyhomogeneous ones.) In the present case these are
simply quantizations of {\em (product-type) symbols} of order $(m,l)$,
$a\in S^{m,l}$,
i.e.\ $\CI$ functions on $\RR^n_z\times\RR^n_\zeta$ such that for all $\alpha,\beta$
$$
|(D_z^\alpha D_\zeta^\beta a)(z,\zeta)|\leq C_{\alpha\beta}\langle
z\rangle^{l-|\alpha|}\langle \xi\rangle^{m-|\beta|}.
$$
Note that as in \cite{Vasy:Propagation-Notes,Vasy:Minicourse}, the second, decay order, uses
the {\em opposite sign convention} than Melrose's original definition
\cite{RBMSpec}; thus, the space $\Psiscc^{m,l}(M)$ gets bigger with
increasing $m,l$. One also has variable order pseudodifferential
symbols and operators. In this paper the relevant order is the second, decay
order, which we must allow to vary, thus $l$ is in $S^{0,0}$; for this
purpose we also need to relax the {\em type} of the symbol estimate
and allow small power (more optimally logarithmic) losses: thus the
estimates for $a\in S^{m,l}_\delta$ are, for fixed $\delta\in (0,1/2)$, which could be taken small
as one wishes for our purposes,
$$
|(D_z^\alpha D_\zeta^\beta a)(z,\zeta)|\leq C_{\alpha\beta}\langle
z\rangle^{l(z,\zeta)-|\alpha|+\delta(|\alpha|+|\beta|)}\langle \xi\rangle^{m-|\beta|+\delta(|\alpha|+|\beta|)}.
$$
The standard concepts, such as the principal symbol, still work; in
the present case it lies in
$S_\delta^{m,l}/S_\delta^{m-1+2\delta,l-1+2\delta}$, and it is
multiplicative, i.e.\ the principal symbol of a product is the product
of the principal symbols, of course with the appropriate orders as usual. Since the choice
of $\delta$ is usually irrelevant, we typically suppress it in the subscripts. One
can thus define the ellipticity, etc., as usual. Then the elements of $\Psiscc^{0,0}$
are bounded operators on all weighted Sobolev spaces $\Hsc^{s,r}$, and
one can define variable order Sobolev spaces (with just $r$ variable for
notational simplicity) by
taking $r_0<\inf r$, and $A\in\Psiscc^{s,r}$ elliptic, and saying
$$
\Hsc^{s,r}=\{u\in\Hsc^{s,r_0}:\ Au\in L^2_\scl\},
$$
with the norm whose square is
$$
\|u\|_{\Hsc^{s,r}}^2=\|u\|_{\Hsc^{s,r_0}}^2+\|Au\|^2_{L^2_\scl};
$$
see \cite{Vasy:Propagation-Notes,Vasy:Minicourse} for details.

It is also important that in addition to principal symbols of
products, we can compute principal symbols of commutators. Concretely,
if $A\in\Psiscc^{s,r}$ and $B\in\Psiscc^{s',r'}$ then
$$
[A,B]\in\Psiscc^{s+s'-1+2\delta,r+r'-1+2\delta}
$$
and with $a$, resp.\ $b$, denoting the principal symbols of $A$,
resp.\ $B$, its principal symbol is the Poisson bracket
$\frac{1}{i}\{a,b\}$ (which, recall, arises from the symplectic
structure on $T^*M^\circ$). In the case of $M$ being the radial compactification of
$\RR^n$, this is simply
$$
\frac{1}{i}\{a,b\}=\frac{1}{i}\sum_{j=1}^n
(\pa_{\zeta_j}a)(\pa_{z_j}b)-(\pa_{z_j} a)(\pa_{\zeta_j}b).
$$
Writing coordinates on the fibers of the scattering cotangent bundle
as $\tau,\mu_1,\ldots,\mu_{n-1}$, sc-dual to the coordinates
$x,y_1,\ldots,y_{n-1}$ discussed above, i.e.\ sc-covectors are written as
$$
\tau\,\frac{dx}{x^2}+\sum_j\mu_j\,\frac{dy_j}{x}
$$
we have
\begin{equation}\begin{aligned}\label{eq:Ham-vf}
\{a,b\}=H_a b=&x\Big((\pa_\tau
a)\big((x\pa_x+\mu\cdot\pa_\mu)b\big)-\big((x\pa_x+\mu\cdot\pa_\mu)a\big)(\pa_\tau
b)\\
&\qquad\qquad+\sum_j\big((\pa_{\mu_j}a)(\pa_{y_j}b)-(\pa_{y_j}a)(\pa_{\mu_j}b)\big)\Big),
\end{aligned}\end{equation}
see \cite[Equation~(5.24)]{RBMSpec}, as follows by a change of
variables computation. Here $H_a$ is called the {\em Hamilton vector
  field} of $a$.

Finally, we need to discuss microlocalization. Since there are two
different important behaviors, in the case of $T^*\RR^n$ these being
$|z|\to\infty$ and $|\zeta|\to\infty$, it is {\em even more useful to
  compactify phase space} than in the usual microlocal analysis
setting, where using homogeneity in dilations of the fibers of the
cotangent bundle is an effective substitute. In the case of $T^*\RR^n$, this
compactified phase space is
$$
\overline{\Tsc^*}\overline{\RR^n}=\overline{\RR^n_z}\times\overline{\RR^n_\zeta}.
$$
i.e.\ we compactify the position and the momentum space separately
using the above radial compactification. Thus $\overline{\Tsc^*}\overline{\RR^n}$ is
the product of two closed balls, and hence it is a manifold with
corners. The two boundary hypersurfaces are $\overline{\RR^n}\times
\pa \overline{\RR^n}$, which is {\em `fiber infinity'}, where standard
microlocal analysis takes place, and $\pa\overline{\RR^n}\times
\overline{\RR^n}$, i.e.\ {\em `base infinity'}; these intersect in the
corner $\pa\overline{\RR^n}\times
\pa\overline{\RR^n}$. The locus of microlocalization is then
$$
\pa \overline{\Tsc^*}\overline{\RR^n}=\pa\overline{\RR^n}\times
\overline{\RR^n}\cup \overline{\RR^n}\times
\pa \overline{\RR^n};
$$
thus the {\em elliptic set}, the {\em characteristic set} and the {\em
  wave front set} are subsets of this.

These notions immediately extend to general manifolds via the local
coordinate identifications. The general compactified phase space is
the fiber-radial compactification $\overline{\Tsc^*} M$ of $\Tsc^*M$;
the locus of microlocalization is its boundary
$$
\pa \overline{\Tsc^*} M=\overline{\Tsc^*}_{\pa M}M\cup\Ssc^*M,
$$
where $\Ssc^*M$ is fiber infinity, i.e.\ the boundary of the fiber
compactification. One can take $x$ as a boundary defining function of
base infinity $\overline{\Tsc^*}_{\pa M}M$ and (with the local
coordinate notation from above)
$\rho_\infty=\langle(\tau,\mu)\rangle^{-1}=(|(\tau,\mu)|^2+1)^{-1/2}$ as a
defining function of fiber infinity, where $|.|$ is the length with
respect to any Riemannian sc-metric. Relating to the above discussion
of variable order spaces, one may have e.g.\ the decay order as smooth
function on $\overline{\Tsc^*}_{\pa M}M$; in order to match with the
previous definition one extends it to a smooth function on all of
$\overline{\Tsc^*}M$, with all potential extensions resulting in exactly the
same Sobolev space as can be seen immediately from the definition.

Note that in view of the vanishing factor $x$ on the
right hand side of \eqref{eq:Ham-vf}, as well as the one order lower than that of $a$ homogeneity of
\begin{equation*}\begin{aligned}
H_a=&x\Big((\pa_\tau
a)(x\pa_x+\mu\cdot\pa_\mu)-\big((x\pa_x+\mu\cdot\pa_\mu)a\big)\pa_\tau
\\
&\qquad\qquad+\sum_j\big((\pa_{\mu_j}a)\pa_{y_j}-(\pa_{y_j}a)\pa_{\mu_j}\big)\Big),
\end{aligned}\end{equation*}
one may want to rescale $H_a$, factoring out this vanishing. In order
to obtain a well-behaved, namely smooth and at least potentially
non-degenerate, vector field on the compactification, one may consider
$$
H_{a,s,r}=x^{r-1}\rho_\infty^{s-1} H_a.
$$
If $a$ is classical this becomes a smooth vector field tangent to the
boundary of the compactified space $\overline{\Tsc^*}M$, and
thus defines a flow on it (the {\em Hamilton flow}); for general $a$ the
vector field is conormal of order $(0,0)$ as a vector field tangent to
the boundary, i.e.\ is in $S^{0,0}\Vb(\overline{\Tsc^*}M)$. Note that
the defining functions we factored out are defined only up to a smooth
positive multiple (smooth on the compact space, thus this means
bounded from above and below by positive constants in particular), and
hence the rescaled vector field is only so defined, but such a change
merely reparameterizes the flow, which is irrelevant for
considerations below. Radial points of $H_a$ are then points on the
boundary of $\overline{\Tsc^*}M$ at which the $H_{a,s,r}$ vanishes (as
a vector in $T \overline{\Tsc^*}M$), thus are critical points of the
flow. Note that at such points $H_{a,s,r}$ need {\em not} vanish as a
b-vector field, so e.g.\ it may have a non-trivial $x\pa_x$ component
(or the analogous statement at fiber infinity);
indeed this non-vanishing is what makes the radial point
non-degenerate: in analytic estimates $x\pa_x$ hitting the weight
$x^m$ of a commutant is what can give a contribution of a definite sign.

After these differentiable manifold structure type discussion in the
sc-category, we briefly discuss the geometry, namely metrics. For the
purposes of the present paper what we need is that the metric $g$ is a
Lorentzian (or more general pseudo-Riemannian) sc-metric for which the
Hamilton flow has a source/sink structure. Thus, we need that there
are submanifolds of $\overline{\Tsc^*}_{\pa M}M$ which are transversal
to $\Ssc^*_{\pa M}M$ and are normal sources 
(meaning normally to the submanifold they are sources) $L_-$, resp.\ sinks $L_+$, of
the Hamilton flow of the dual metric function, $G$ (the principal
symbol of $\Box_g$). The definition of a {\em non-trapping sc-metric} $g$ is
then that $g$ is a sc-metric (of some signature) such that all
integral curves of $H_G$ inside the characteristic set $\{G=0\}$ at
$\Ssc^* M$ (i.e.\ fiber infinity), except those contained
in $L_\pm$, tend to $L_+$ (indeed, necessarily to $L_+\cap\Ssc^*M$) in the forward
and $L_-$ in the backward direction. We call a sc-metric {\em non-trapping
at energy $\lambda$} if it is non-trapping in the sense above, and in
addition, all
integral curves of $H_G$ inside the $\lambda$ characteristic set $\{G=\lambda\}$ at
$\overline{\Tsc^*}_{\pa M}M$ (i.e.\ base infinity), except those contained
in $L_\pm$, tend to $L_+$ (indeed, necessarily to
$L_+\cap\overline{\Tsc^*}_{\pa M}M$) in the forward
and $L_-$ in the backward direction. 

The class of non-trapping Lorentzian sc-metrics is a much larger class of metrics than that of Lorentzian
scattering metrics introduced in \cite{Baskin-Vasy-Wunsch:Radiation},
for the latter class also demands that $L_\pm$ be (scattering) conormal bundles of
submanifolds $S_\pm$ of $\pa M$ at which the metric has a certain
model form, generalizing that of the Minkowski metric on the radial
compactification, and also that $S_\pm$ are (non-degenerate) zero sets of a certain
function $v\in\CI(\pa M)$, and $\{v>0\}$ has two connected components
$C_\pm$ (`spherical caps' in Minkowski space, and though they may not
be spherical in any sense, we continue calling them so in general) with $S_\pm$ as their respective
boundaries, while $\{v<0\}=C_0$ has two boundary components
$S_\pm$. The `non-trapping at energy $\lambda$' condition played no
role in \cite{Baskin-Vasy-Wunsch:Radiation}, since that paper
considered the wave equation ($\lambda=0$) and rescaled the wave
operator to a b-operator, for which the non-trapping condition is
exactly the above sc-non-trapping condition; this rescaling is
possible and the non-trapping claim holds since
the dual metric $G$ is necessarily a homogeneous quadratic polynomial on the
fibers of $\Tsc^*M$, unlike $G-\lambda$ for non-zero $\lambda$.

\section{Fredholm theory and essential self-adjointness}
In our proof of the main theorem we focus on the Lorentzian case of
signature $(1,n-1)$ for
the sake of being definite in terminology;
the general pseudo-Riemannian case
barely differs, except that the {\em only} reasonable problems are the
Feynman and anti-Feynman problems, but they are also {\em the only ones
  that matter} below. In all of our discussions below we assume that
the metrics are {\em non-trapping sc-metrics}.

In order to get started, we first
recall that (assuming $M$ is connected -- otherwise the statement is
for each connected component) for $\lambda>0$ the Klein-Gordon operator
$P=\Box_g-\lambda$ has four Fredholm problems, see
\cite{Vasy:Positive, Vasy:Minicourse, Vasy:Propagation-Notes},
corresponding to the characteristic set having two connected
components. (This follows from the characteristic set $\{\zeta:\
G(\zeta,\zeta)=\lambda\}$, $G=g^{-1}$, having two connected components fiberwise:
this fiberwise
characteristic set is the two-sheeted hyperboloid.) Indeed, in each connected component of the characteristic
set one can choose the direction in which one propagates estimates for
$P$, and then for $P^*$ using dual spaces one propagates the estimates
in the opposite direction, resulting in $2^2$
possibilities. Concretely, these are the retarded,
advanced, Feynman and anti-Feynman, Fredholm problems; the direction
of propagation is encoded by the use of appropriate weighted Sobolev type
spaces. Concretely, these are based on variable order Sobolev spaces
$\Hsc^{s,r}$, where $s$ is constant, $r$ variable, a function on
$\overline{\Tsc^*}_{\pa M}M$, monotone along the rescaled Hamilton
flow, and satisfies the inequalities $r>-1/2$ at the radial points
from which estimates are propagated, $r<-1/2$ at the radial points to
which estimates are propagated,
and
$$
\cY^{s-1,r+1}=\Hsc^{s-1,r+1},\ \cX^{s,r}=\{u\in\Hsc^{s,r}:\ (\Box_g-\lambda) u\in\Hsc^{s-1,r+1}\},
$$
with
$$
P=\Box_g-\lambda:\cX^{s,r}\to\cY^{s-1,r+1}
$$
Fredholm.
Corresponding to the above discussion, this fact relies on the non-trapping nature of the bicharacteristic
flow within the characteristic set, which has two parts: the part at
fiber infinity, in $\Ssc^*M$, which is independent of
$\lambda$, and the part at spatial infinity, $\overline{\Tsc^*}_{\pa
  M}M$, which does depend on $\lambda$. In particular, these
non-trapping conditions are perturbation stable, and hold for the
Minkowski (as well as translation invariant pseudo-Riemannian on $\RR^n$) metrics. Note that the choice of $s,r$ with $r$ satisfying the above
constraints is irrelevant; and the nullspace automatically lies in the
intersection of all these spaces; thus in particular for elements $u$
of the nullspace of $P$, $\WFsc(u)$ is a subset of the radial points towards
which the estimates are propagated.

If $\lambda<0$ (and $n=\dim M\geq 3$), then due to the behavior of the characteristic set at
base infinity, the characteristic set has only one connected
component (the characteristic set $\{\zeta:\
G(\zeta,\zeta)=\lambda\}$ has one connected component fiberwise: this fiberwise
characteristic set is the one-sheeted hyperboloid.), and correspondingly only two of these problems remain:
Feynman and anti-Feynman problems. Note that the Cauchy problem (or
retarded/advances problems) is
still solvable, but the solution will typically grow exponentially,
thus it does not exist as far as polynomially weighted Sobolev spaces
(our world in this paper) are concerned.
(For $\lambda=0$ we still have the four problems, but in weighted
b-Sobolev spaces, see \cite{Gell-Redman-Haber-Vasy:Feynman}, which we
do not discuss here.)

Now, if $\lambda$ is made complex, of course the usual principal
(and even the subprincipal!) symbol are not affected, and correspondingly
estimates at fiber infinity, $\Ssc^*M$, are unchanged over
$M^\circ$. However, the estimates at $\overline{\Tsc^*}_{\pa M}M$
become more delicate.

Namely, in this region one has a non-real principal symbol (since
$\lambda$ is part of it). Thus, by the usual propagation estimates
(these are the ones used for {\em `complex absorption'}) one can propagate estimates in the
{\em forward} direction along the $H_{\re p}$ flow when $\im\lambda\geq 0$, and
in the {\em backward} direction when $\im\lambda\leq 0$. (See
\cite[Section~5.4.5]{Vasy:Minicourse} and
\cite{Vasy:Propagation-Notes},
as well as the usual microlocal analysis version in \cite[Section~2.5]{Vasy-Dyatlov:Microlocal-Kerr}.)  Of course,
the operator is elliptic at {\em finite} points (not at $\Ssc^*M$) of
$\overline{\Tsc^*}_{\pa M}$ when $\im\lambda>0$; the point is that the
estimates work at the corner (fiber infinity at $\pa M$), and they
work {\em uniformly} in $\lambda$ even as $\im\lambda\to 0$. This propagation works for
any $s,r$ ($s$ a priori relevant only when one is at fiber infinity at
$\pa M$), including variable $r$, when $r$ is monotone decreasing in
the direction in which the estimates are propagated.

Notice that corresponding to the ellipticity at finite points, for these estimates, as well as the ones below, the only
relevant non-trapping condition is the basic one concerning
bicharacteristics at fiber infinity, $\Ssc^*M$, i.e.\ the
`non-trapping at energy $\lambda$' condition is only relevant when one
wants to let $\lambda$ to the real axis, as we do below for the
limiting absorption principle (but not for the essential
self-adjointness discussion), and then the relevant
condition is non-trapping at {\em the limiting energy $\lambda$}.

Most importantly
though one needs to get radial point estimates, however. For real
$\lambda\neq 0$
these are two estimates, see \cite[Section~5.4.7]{Vasy:Minicourse} and
\cite{Vasy:Propagation-Notes}, as well as the usual microlocal analysis version in \cite[Section~2.4]{Vasy-Dyatlov:Microlocal-Kerr}. The first, high `regularity' (where here this
means decay), which gives a `free' estimate at the radial point, is of the form
$$
\|Qu\|_{\Hsc^{s,r}}\leq C(\|Q_1 Pu\|_{\Hsc^{s-1,r+1}}+\|Q_1 u\|_{\Hsc^{s',r'}}+\|u\|_{\Hsc^{M,N}})
$$
when $r>r'>-1/2$, $s,s',M,N$ arbitrary (one usually considers $s',M,N$ very large
and negative; these are background error terms with relatively compact
properties), $Q$ elliptic at the radial set, with wave front set in
a small neighborhood, $Q_1$ elliptic on $\WF'(Q)$ and
bicharacteristics from all points in the intersection of $\WF'(Q)$ and
the characteristic set tend to $L$ in the appropriate
(forward/backward) direction depending on the sink/source nature,
remaining in $\Ell(Q_1)$.
The second, low `regularity' (where again here this
means decay), which allows one to propagate estimates into the radial
point from a punctured neighborhood, is of the form
$$
\|Qu\|_{\Hsc^{s,r}}\leq C(\|Q_2u\|_{\Hsc^{s,r}}+\|Q_1 Pu\|_{\Hsc^{s-1,r+1}}+\|Q_1 u\|_{\Hsc^{s',r'}}+\|u\|_{\Hsc^{M,N}})
$$
when $r<-1/2$, $s,s',r',M,N$ arbitrary (one considers $s',r',M,N$ very large
and negative), $Q$ elliptic at the radial set, with wave front set in
a small neighborhood, $Q_1$ elliptic on $\WF'(Q)$ and
bicharacteristics from all points in $\WF'(Q)$ intersected with the
characteristic set which are not in $L$, tend to $L$ in the appropriate
(forward/backward) direction depending on the sink/source nature,
remaining in $\Ell(Q_1)$, and intersect $\Ell(Q_2)$ at some point in
the opposite direction along the flow, still remaining in $\Ell(Q_1)$.

Now allowing $\lambda$ complex, say $\im\lambda\geq 0$, one can only propagate estimates in
the forward direction along the $H_{\re p}$-flow, and correspondingly one
obtains the high regularity estimates only at the sources, the low
regularity ones at the sinks (with sources and sinks reversed for
$\im\lambda\leq 0$). These estimates in fact become {\em stronger} than
the ones above, cf.\ the complex absorption arguments in
\cite[Section~5.4.5]{Vasy:Minicourse} and \cite{Vasy:Propagation-Notes}, namely one can in addition control a
term $\|Qu\|_{\Hsc^{s-1/2,r+1/2}}$, i.e.\ one that is stronger in the
sense of decay (though not differentiability) than that on $Qu$
above. This results from an extra term $\langle \check A^*\im\lambda
\check Au,u\rangle=\im\lambda\|\check A u\|^2$ in the estimate, in addition to the commutator
terms, $\langle[\check A^*\check A,\Box]u,u\rangle$, with $\check A\in
\Psiscc^{m'/2,l'/2}$, $s=(m'+1)/2$, $r=(l'-1)/2$. Thus the estimates
are
\begin{equation}\label{eq:high-reg-rad-point-cmplx}
\|Qu\|_{\Hsc^{s,r}}+\im\lambda \|Qu\|_{\Hsc^{s-1/2,r+1/2}}\leq C(\|Q_1 Pu\|_{\Hsc^{s-1,r+1}}+\|Q_1 u\|_{\Hsc^{s',r'}}+\|u\|_{\Hsc^{M,N}})
\end{equation}
$r>r'>-1/2$, $s,s',M,N$ arbitrary, and
\begin{equation}\begin{aligned}\label{eq:low-reg-rad-point-cmplx}
\|Qu\|_{\Hsc^{s,r}}&+\im\lambda\|Qu\|_{\Hsc^{s-1/2,r+1/2}}\\
&\leq C(\|Q_2u\|_{\Hsc^{s,r}}+\|Q_1 Pu\|_{\Hsc^{s-1,r+1}}+\|Q_1 u\|_{\Hsc^{s',r'}}+\|u\|_{\Hsc^{M,N}})
\end{aligned}\end{equation}
when $r<-1/2$, $s,s',r',M,N$ arbitrary.

Now, taking $r$ with $-1/2<r$ at the sources, $-1/2>r>-1$ at the
sinks, monotone along the flow, $s>1/2$, this in particular gives:

\begin{prop}(See \cite[Section~5.4.8]{Vasy:Minicourse} for the real
  $\lambda$ version.)
Suppose $\lambda\neq 0$. Then for $s,r$ as above corresponding to
either the Feynman spaces ($\im\lambda\geq 0$) or anti-Feynman spaces
($\im\lambda\leq 0$),
the operator
$$
P:\cX^{s,r}\to\cY^{s-1,r+1}
$$
is Fredholm, with
$$
\cY^{s-1,r+1}=\Hsc^{s-1,r+1},\ \cX^{s,r}=\{u\in \Hsc^{s,r}:\ Pu\in\Hsc^{s-1,r+1}\}.
$$
\end{prop}

One can interpret the estimates
\eqref{eq:high-reg-rad-point-cmplx}-\eqref{eq:low-reg-rad-point-cmplx},
as well as the analogous real principal type estimates in the
characteristic set between the radial points as additional regularity
estimates giving that in fact
\begin{equation}\label{eq:X-sr-refined}
\cX^{s,r}=\{u\in \Hsc^{s,r}\cap \Hsc^{s-1/2,r+1/2}:\ Pu\in\Hsc^{s-1,r+1}\}.
\end{equation}
In particular, this lets one solve $Pu=f$, $f\in\dCI(M)$, up to finite
dimensional obstacles, namely one gets that the solution
$u$ (which exists in the complement of a finite dimensional subspace)
is almost in $L^2$, namely $u\in\Hsc^{0,-\ep}$ for all
$\ep>0$. Indeed, we have that if $Pu=f$, $f\in\dCI(M)$, with
$u\in\cX^{s,r}$ as above, then $u\in
\Hsc^{\tilde s-1/2,\tilde r+1/2}$ for all $\tilde s$ and for all
$\tilde r<-1/2$, thus in $\Hsc^{\infty,-\ep}$ for all $\ep>0$.

This is not quite sufficient, however, since we want to conclude $u\in
L^2_\scl$, and also that there are no finite codimension issues (i.e.\ we
have invertibility and not just Fredholmness) so, for $\im\lambda>0$, one needs to do a borderline
estimate, with $r=-1/2$ at the sink (everywhere else one is in $L^2$
already), which corresponds to $l'=0$. Note that such an estimate cannot
work when $\im\lambda=0$, and thus cannot be uniform in $\im\lambda$
when $\im\lambda>0$. The key point is that in this case the commutator
$[\check A^*\check A,\Box]$ will have principal symbol at $\Ssc^*M$
for which the normally main term (arising from the weight) vanishes at
$L$.

It suffices for us to consider $m'=l'=0$, in which case the
situation is very simple: we will take $\check A$ to be microlocally
the identity near the sinks, i.e.\ to have $\WFsc'(\Id-\check A)$ disjoint
from the sink. (Such microlocalizers play an important role in the
proof of asymptotic completeness in the $N$-body setting; a partially
microlocal version is the work of Sigal and Soffer
\cite{Sigal-Soffer:N} and Yafaev \cite{Yaf}, see \cite{Vasy:Geometry} for a discussion.)

Since $\WFsc(u)$ is in the sink when $f\in\dCI(M)$, for
$u\in\Hsc^{s',r'}$ (with e.g.\ $r'=r+1/2$ from above, so $<0$ but
close to $0$ allowed) the pairing $\langle u,[\check A^*\check
A,\Box]u\rangle$ makes sense
(and a regularized version remains bounded: the regularizer gives the
correct sign as it behaves exactly the same way as if one had a more
decaying weight, i.e.\ as if $l'<0$) if $2r'-l'+2\geq 0$ and
$2s'-m'-1\geq 0$, which holds with $l'=m'=0$ if $r'<0$ is close to $0$
and $s'=1$, say. Then the $\im\lambda$ term gives an estimate for
$\|\check A u\|^2$, which is an estimate for $u$ in
$\Hsc^{m'/2,l'/2}=L^2_\scl$ as desired. In particular, if
$u\in\cX^{s,r}$ and $Pu\in\dCI(M)$ then $u\in L^2_\scl$.

\begin{lemma}\label{lemma:trivial-nullspace}
Suppose that $\im\lambda>0$, and consider the Feynman Fredholm 
problem $\cX^{s,r}\to\cY^{s-1,r}$. 
Then $\Ker P$ and $\Ker P^*$ (on the dual space) are trivial.

Analogous statements hold for $\im\lambda<0$ for the anti-Feynman
Fredholm problem.
\end{lemma}

\begin{proof}
We have already seen that elements of $\Ker P$ and $\Ker P^*$ lie in
$L^2_\scl$. Thus, formally the lemma
follows from
$$
0=\langle Pu,u\rangle-\langle u,Pu\rangle=\langle
(P-P^*)u,u\rangle=-2\imath\langle\im\lambda u,u\rangle=-2\imath\im\lambda\|u\|^2,
$$
but the issue is that $\langle P^*u,u\rangle$ does not actually make
sense a priori due to the too weak a priori differentiability of $u$
when the unweighted spaces are used (all we know is that $u\in
L^2_\scl$, so $P^* u\in\Hsc^{-2,0}$ only, unless we use
$P-P^*\in\Psiscc^{0,0}$, but even then we need to justify the
integration by parts (because the adjoint a priori puts us in dual spaces)!), so we need to have a more careful, if standard, regularization
argument.

Namely, we take $\Lambda_t\in\Psiscc^{-\infty,0}$, $t\in[0,1]$ such
that the family is uniformly bounded in $\Psiscc^{0,0}$ and converges
to $\Id$ in $\Psiscc^{\ep,0}$, $\ep>0$, as $t\to 0$, and thus strongly on
$L^2_\scl$. Then we have for $t>0$, if $u\in L^2_\scl$ and $Pu\in L^2_\scl$,
\begin{equation*}\begin{aligned}
0&=\langle Pu,u\rangle-\langle u,Pu\rangle=\lim_{t\to 0}\big(\langle \Lambda_t
Pu,u\rangle-\langle \Lambda_t u,Pu\rangle\big)\\
&=\lim_{t\to 0}\big(\langle \Lambda_t
Pu,u\rangle-\langle P^*\Lambda_t u,u\rangle\big)\\
&=\lim_{t\to 0}\big(\langle \Lambda_t
Pu,u\rangle-\langle \Lambda_t P^* u,u\rangle-\langle [P^*,\Lambda_t]
u,u\rangle\big)\\
&=\lim_{t\to 0}\big(\langle \Lambda_t
Pu,u\rangle-\langle \Lambda_t (P+2\imath\im\lambda) u,u\rangle-\langle [P^*,\Lambda_t]
u,u\rangle\big)\\
&=-2\imath\im\lambda\langle u,u\rangle -\lim_{t\to 0}\langle [P^*,\Lambda_t]
u,u\rangle.
\end{aligned}\end{equation*}
Now, $P^*\in\Psiscc^{2,0}$, so $[P^*,\Lambda_t]$ is uniformly bounded
in $\Psiscc^{1,-1}$, and it converges to $[P^*,I]=0$ in
$\Psiscc^{1+\ep,-1+\ep}$ for $\ep>0$, thus strongly as a bounded operator
$\Hsc^{1/2,-1/2}\to\Hsc^{-1/2,1/2}$. Correspondingly, if in addition
$u\in\Hsc^{1/2,-1/2}$, then the last term vanishes, and we conclude
that $u=0$. But we have seen that in the Feynman spaces this holds,
namely $u\in\Hsc^{\infty,-\ep}$ for all $\ep>0$, so we conclude that
$\|u\|^2=0$ and thus $u=0$ as well.

The analogous argument also holds for $P^*$ on the anti-Feynman space,
which proves that $P^*$ is also injective.
\end{proof}

\begin{cor}
Suppose $\im\lambda\neq 0$.
The operator
$P:\cX^{s,r}\to\cY^{s-1,r+1}$ is indeed invertible (not just Fredholm)
and moreover we have for $f\in\dCI(M)$ that $u=P^{-1}f\in L^2_\scl$ as
well.
\end{cor}

We take
$$
D=\{u\in \Hsc^{1,-1/2}\cap L^2_\scl: \Box_g u\in L^2_\scl\}.
$$
Actually $\Hsc^{1,-1/2}$ in this definition could be replaced by
$\Hsc^{s',r'}$ for any $s'\in[1,2]$, $r'\in[-1/2,0)$, as is immediate
from the following argument; slightly increasing the requirements on
the above choice of $r$ is then needed for the statement of the next sentence.
Then for either sign of $\im\lambda$, and corresponding choices of
$s,r$ (only $r$ depends on $\im\lambda$) as above with the slightly
stronger requirements $s\geq 2$, while $r>-1$, with $r<-1/2$ at the
sinks, we have $\cX^{s,r}\cap L^2_\scl\subset D$ since $r$ takes values in
$(-1,\infty)$, so $r+1/2$ in $(-1/2,\infty)$ and since
$(\Box_g-\lambda)u\in\Hsc^{s-1,r+1}\subset L^2_\scl$ implies $\Box_g u\in L^2_\scl$. Now, $D$ is a Hilbert
space. Moreover, $\dCI(M)$ is dense
in $D$  since
using $\tilde\Lambda_t\in\Psiscc^{-\infty,-\infty}$ uniformly bounded
in $\Psiscc^{0,0}$, converging to $\Id$ in $\Psiscc^{\ep,\ep}$ for all
$\ep>0$, we have  $[\Box_g,\tilde\Lambda_t]$ uniformly bounded in
$\Psiscc^{1,-1}$, converging to $0$ in $\Psiscc^{1+\ep,-1+\ep}$, thus
strongly as a map $\Hsc^{1,-1}\to\Hsc^{0,0}=L^2_\scl$. Hence $\dCI(M)\ni\tilde\Lambda_t u\to u$ in $L^2_\scl$, as
well as in $\Hsc^{1,-1/2}$ and $\Box_g\tilde\Lambda_t
u=\tilde\Lambda_t\Box_g u+[\Box_g,\tilde\Lambda_t]u\to \Box_g u$ in
$L^2_\scl$ as $D\subset\Hsc^{1,-1}$. (Notice that here the argument
goes through with $(1,-1/2)$ replaced by $(s',r')$ in the Sobolev
order in the definition of $D$, corresponding to the remark after the
definition: the density statement becomes {\em easier} then as there
needs to be {\em less} of a gain for the commutator.)
See
\cite[Appendix~A]{Melrose-Vasy-Wunsch:Corners} for a more general
discussion on spaces like $D$; in the present context
\cite[Section~4]{Vasy:Minicourse} would be the relevant setting, but
the present statement is not proved there, though the proof of
\cite[Lemma~A.3]{Melrose-Vasy-Wunsch:Corners} applies, mutatis
mutandis.

Furthermore, $\Box_g$ is symmetric on this domain since
\begin{equation*}\begin{aligned}
\langle\Box_g u,u\rangle-\langle u,\Box_g u\rangle&=\lim_{t\to
  0}\langle\tilde\Lambda_t\Box_g u,u\rangle-\langle \tilde\Lambda_t
u,\Box_g u\rangle\\
&=\lim_{t\to 
  0}\langle\tilde\Lambda_t\Box_g u,u\rangle-\langle \Box_g\tilde\Lambda_t 
u,u\rangle\\
&=\lim_{t\to 
  0}\langle\tilde\Lambda_t\Box_g u,u\rangle-\langle \tilde\Lambda_t \Box_g
u,u\rangle-\langle [\Box_g,\tilde\Lambda_t]
u,u\rangle\\
&=-\lim_{t\to 
  0}\langle [\Box_g,\tilde\Lambda_t]
u,u\rangle.
\end{aligned}\end{equation*}
Indeed, as noted above $[\Box_g,\tilde\Lambda_t]$ is uniformly bounded in
$\Psiscc^{1,-1}$, converging to $0$ in $\Psiscc^{1+\ep,-1+\ep}$, thus
strongly as a map $\Hsc^{1/2,-1/2}\to\Hsc^{-1/2,1/2}$, so for
$u\in D\subset \Hsc^{1,-1/2}\subset\Hsc^{1/2,-1/2}$ the
right hand side tends to $0$ and we have the desired conclusion of
symmetry. (This immediately implies the general $s',r'$ case since the
space $D$ becomes a priori smaller.)

Thus, $\Box_g:D\to L^2_\scl$ is a continuous map, $\CI_c(M^\circ)$ is dense in $D$
(by virtue of $\dCI(M)$ being so), and $\Box_g$ is a symmetric
operator. In order to prove that $\Box_g$ is essentially self-adjoint,
it suffices to prove that for $\lambda\notin\RR$, $\Box_g-\lambda$ has
a dense range in $L^2_\scl$. But $\dCI(M)$ is dense in $L^2_\scl$, so
it suffices to show that for $\im\lambda\neq 0$ and $f\in\dCI(M)$ there exists $u\in D$ such
that $(\Box_g-\lambda)u=f$. But we have seen above that under these
conditions there exists $u\in\cX^{s,r}$ such that
$(\Box_g-\lambda)u=f$, and moreover $u\in L^2_\scl$, so as
$\cX^{s,r}\cap L^2_\scl\subset D$ in view of \eqref{eq:X-sr-refined}, the desired conclusion
follows.
(Here the limitations on $s',r'$ are strongly relevant and required for the
generalized version of $D$ discussed above, together with
the corresponding strengthening of the requirements on $s,r$.) This proves that $\Box_g$ is essentially selfadjoint on
$\CI_c(M^\circ)$, namely proves Theorem~\ref{thm:ess-sa}.

\section{The limiting absorption principle}
The limiting absorption principle is an immediate consequence of our
discussion. Namely, under the assumption of $g$ being non-trapping at
energy $\lambda$ for the limiting $\lambda$ (or interval of
$\lambda$'s, if one wishes), the estimates for $\Box_g-\lambda$ on the Feynman
spaces are uniform in $\im\lambda\geq 0$, and similarly on the
anti-Feynman spaces in $\im\lambda\leq 0$; and indeed, for
$\lambda\in\RR$, $\Box_g-\lambda$ is Fredholm on either one of these
spaces. Furthermore, when $\im\lambda\neq 0$, the operator is
invertible, thus index $0$, and this is stable under perturbations
(even of the kind we discussed), cf.\
\cite[Section~2.7]{Vasy-Dyatlov:Microlocal-Kerr}, which also discusses
continuity in the weak operator topology. In particular, the limit
{\em is} the (anti-)Feynman propagator, up to finite dimensional
nullspace issues on the limiting space.
Thus,

\begin{thm}\label{thm:LAP}
Suppose $\lambda\in\RR\setminus \{0\}$, $g$ is a non-trapping
sc-metric which is non-trapping at energy $\lambda$ and $\Box_g-\lambda$ has
trivial nullspace on either
the Feynman or the anti-Feynman function spaces. Then
$\lim_{\ep\to 0}(\Box_g-(\lambda\pm \imath\ep))^{-1}$ exist in the
weak operator topology on the Feynman ($+$), resp.\ anti-Feynman ($-$)
function spaces, and {\em is} the Feynman, resp.\ anti-Feynman
propagator, i.e.\ the inverse of $\Box_g-\lambda$ on the appropriate
function spaces.
\end{thm}

\begin{rem}
The argument of \cite[Proposition~3.1]{Vasy-Wrochna:QFT} applies with minor notational
changes (corresponding to the b-setting employed there and the sc-setting
employed here) to prove that the primed wave front set of the Schwartz kernel
of $(\Box_g-(\lambda\pm \imath 0))^{-1}$ is in the backward/forward
flowout of the diagonal of the cotangent bundle $T^*M^\circ$ over the
interior of $M$.
\end{rem}

Of course, it is still a question whether the nullspace of
$\Box_g-\lambda$ is trivial; the set of $\lambda$ for which it is, is
necessarily open by stability. Again, this stability is true even for the relatively drastic
kind of perturbations we have which change the domain space (since the space
depends strongly on $\lambda$ via the condition
$(\Box_g-\lambda)u\in\Hsc^{s-1,r+1}$); the point is that the relevant
estimates {\em on fixed spaces}, with {\em fixed relatively compact
  error terms} are perturbation stable. Interestingly, cf.\ the discussion of
\cite[Section~4]{Vasy:Positive}, adopting arguments of Isozaki
\cite{IsoRad} from $N$-body scattering as done in \cite[Proof of Proposition~17.8]{Vasy:Propagation-2},
which are valid after minor modification in this setting as we discuss
below, any element of the nullspace
of $\Box_g-\lambda$ in either the Feynman or the anti-Feynman spaces
in fact lies in $\dCI(M)$:

\begin{prop}
On both the Feynman and anti-Feynman function spaces the nullspace of
$P=\Box_g-\lambda$ is a subspace of $\dCI(M)$.
\end{prop}

\begin{proof}
The arguments following \cite{IsoRad} rely on using
the commutant
$$
\chi_\ep(x)=\ep^{-2r-1}\int_0^{x/\ep}\phi(s)^2 s^{-2r-2}\,ds,
$$
where $\phi\in\CI(\RR)$ is such that $\phi=0$ on $(-\infty,1]$, $1$ on
$[2,\infty)$, and where $r\in(-1/2,0)$. Notice that $\chi_\ep$ is supported in $M^\circ$ (namely
in $x\geq\ep$) and
$$
[x^2\pa_x,\chi_\ep]=x^2\pa_x(\chi_\ep)=x^{-2r}\phi(x/\ep)^2.
$$
Thus, while the family $\{\chi_\ep:\ \ep\in(0,1)\}$ (considered as a
family of multiplication operators) is not (uniformly)
bounded in
any symbol space, as (using $s'=\ep s$) for any $l$
$$
\sup x^l\chi_\ep\geq
\chi_\ep(1)\geq \int_{2\ep}^1\phi(s'/\ep)^2 (s')^{-2r-2}\,ds'=\frac{1}{2r+1}((2\ep)^{-2r-1}-1)\to\infty
$$
as $\ep\to 0$ since $-2r-1<-2(-1/2)-1=0$, its commutator with
$x^2\pa_x$ is bounded in symbols of order $2r$. This gives, by \eqref{eq:Ham-vf}, that
the principal symbol of $\imath[P,\chi_\ep]$ is, with $p$ denoting the
principal symbol of $P$,
$$
H_p\chi_\ep=(\pa_\tau p) x^2\pa_x (\chi_\ep);
$$
note that the $x^2\pa_x$ component of $H_p$ is exactly $\pa_\tau p$,
so at a source, resp.\ sink, manifold of the boundary, $\pa_\tau p<0$,
resp.\ $\pa_\tau p>0$, i.e.\ at such a manifold this commutator has a
definite sign. While the lower order terms, which here means just the
$0$th order terms as we have a differential operator in $P$, involve
further derivatives of $\chi_\ep$, they {\em only} involve at least
first derivatives of $\chi_\ep$ and thus the lower order terms will
also be bounded in symbols of order $0,2r-1$. Thus,
\begin{equation}\label{eq:pos-comm-op}
\imath[\Box_g-\lambda,\chi_\ep]=\pm\phi(x/\ep) (B^*B+E)\phi(x/\ep)+F_\ep,
\end{equation}
where $B\in\Psiscc^{1/2,r}$, with principal symbol elliptic at the
sources/sinks (depending on the choice of $\pm$, with $+$ for sinks), $E\in\Psiscc^{1/2,r}$ having
disjoint wave front set from these, and $\{F_\ep:\ \ep\in(0,1)\}$ is
uniformly bounded in $\Psiscc^{0,2r-1}$.

Now consider $u\in \Ker P$ on
$\cX^{s',r'}$, where $s'$ may be taken arbitrarily high and $r'$
arbitrarily high except in a neighborhood of the source/sink in
accordance with the sign in $\pm$ above, where $r'\in(-1,-1/2)$
(`arbitrarily high' is in
the sense that the nullspace is {\em independent of such choices}). Then
$\langle \phi(\cdot/\ep)E\phi(\cdot/\ep)u,u\rangle$ remains bounded as
on $\WFsc'(E)$, $u$ is microlocally in $\Hsc^{\infty,\infty}=\dCI(M)$,
while $\langle F_\ep u,u\rangle$ also remains bounded since
$u\in\Hsc^{\infty,-1/2-\delta'}$ for all $\delta'>0$ and $2r-1<-1$, so
one can choose $\delta'>0$ with
$2(-1/2-\delta')-(2r-1)>0$. On the other hand, for $\ep>0$,
$$
\langle
\imath[\Box_g-\lambda,\chi_\ep]u,u\rangle=\langle \imath\chi_\ep
u,(\Box_g-\lambda) u\rangle-\langle \imath(\Box_g-\lambda)u,\chi_\ep u\rangle=0
$$
since $\chi_\ep$ is compactly supported in $M^\circ$, so the
integration by parts is justified.
Correspondingly, one deduces that
$B\phi(\cdot/\ep)u$ is uniformly bounded in $L^2_\scl$, and thus by
the standard weak-* convergence argument $Bu\in L^2_\scl$, proving that
even at the source/sink where we did not have a priori knowledge of
membership of $u$ in a subspace of $\Hsc^{\infty,-1/2}$, in fact,
$u\in\Hsc^{\infty,r}$ for all $r\in(0,-1/2)$. Then the standard radial
point estimate, see \cite{Vasy:Propagation-Notes, Vasy:Minicourse},
implies that in fact $u$ is microlocally in $\Hsc^{\infty,\infty}$
even there; in combination with the other a priori knowledge, we
conclude that $u\in \dCI(M)$.
\end{proof}

\begin{rem}
Notice that this argument used crucially that $Pu=0$; since $\chi_\ep$
is not uniformly bounded on any weighted Sobolev space, $\langle \chi_\ep
u,(\Box_g-\lambda) u\rangle$, $\langle (\Box_g-\lambda)u,\chi_\ep
u\rangle$ would not remain bounded as $\ep\to 0$ otherwise even if,
say, $Pu\in\dCI(M)$.

Also notice that the argument crucially relies that we are taking
either the Feynman or the anti-Feynman space, so the points at which
we do not have a priori decay are either all sources or all sinks,
thus there is a single definite sign in \eqref{eq:pos-comm-op},
arising for the common source, or sink, nature of them. For
other Fredholm problems, the elements of the nullspace are not
necessarily in the `trivial space', $\dCI(M)$.
\end{rem}

Thus, the absence of embedded eigenvalues
depends on a unique continuation argument at infinity, namely that the
rapid decay (infinite order vanishing) at $\pa M$ of an element of
$\Ker P$ implies its vanishing nearby.

In the case of non-trapping
Lorentzian scattering metrics (possibly long range), as in
\cite{Baskin-Vasy-Wunsch:Radiation, Baskin-Vasy-Wunsch:Long-range}, if one assumes that the there is a
boundary defining function $\rho$ of $M$ such that, say, near the past
`spherical cap' $\overline{C_-}$, $\frac{d\rho}{\rho^2}$, is timelike
(which for instance is true on perturbations of Minkowski space), for $\lambda>0$
energy estimates imply that, being an element of $\dCI(M)$, an element of this nullspace
vanishes identically at first near $\overline{C_-}$, and then the
non-trapping condition implying global hyperbolicity, see
\cite[Section~5]{Hintz-Vasy:Semilinear} in this setting for this implication, vanishes globally, so
the nullspace is indeed trivial.

An analogous conclusion holds by a Wick rotation argument, see
\cite{Gell-Redman-Haber-Vasy:Feynman}, for the Minkowski metric, as
well as pseudo-Riemannian translation invariant metrics, and again
the perturbation stability implies that the conclusion also holds for
their perturbations in the sc-category.

We finally remark that the $\lambda=0$ Fredholm problem was studied in
\cite{Gell-Redman-Haber-Vasy:Feynman}; one can also discuss the
limiting absorption principle there, under somewhat stronger
conditions than we needed here, but we defer it to future work.

\bibliographystyle{plain}
\bibliography{sm}

\end{document}